\documentclass[a4paper,11pt]{amsart}
\usepackage{amsthm,amsfonts,amsbsy,amssymb,amsmath,amscd}
\usepackage[all]{xy}
\usepackage[colorlinks=true]{hyperref}
\usepackage{pgf, tikz, tikz-cd}

\title{On the pluricanonical map of projective 3-folds of general type with $P_3 \geq 2$}
\author{Yong Hu, Jianshi Yan}

\address{\rm School of Mathematical Sciences, Shanghai Jiao Tong University, Shanghai 200240, China}
\email{yonghu@sjtu.edu.cn}

\address{\rm Department of Mathematics, Northeastern University, Shenyang 110819, China}
\email{yanjs@mail.neu.edu.cn}

\newcommand{\bQ}{{\mathbb Q}}
\newcommand{\bP}{{\mathbb P}}
\newcommand{\roundup}[1]{\ulcorner{#1}\urcorner}
\newcommand{\rounddown}[1]{\llcorner{#1}\lrcorner}

\newcommand\OX{{\mathcal{O}_X}}

\newcommand{\lsgeq}{\succcurlyeq}

\newcommand{\Mov}{\text{Mov}}

\newcommand{\simQ}{\sim_{\mathbb{Q}}}

\newtheorem{thm}{Theorem}[section]
\newtheorem{lem}[thm]{Lemma}

\newtheorem{prop}[thm]{Proposition}

\theoremstyle{definition}
\newtheorem{defn}[thm]{Definition}

\newtheorem{exmp}[thm]{Example}

\theoremstyle{remark}

\newtheorem{proof of thm}{\bf Proof of Theorem \ref{modified lemma for 3-fold}}
\begin{document}

\begin{abstract}
  We prove that for all nonsingular projective 3-folds of general type with third plurigenus $P_3 \geq 2$, the pluricanonical map $\varphi_m$ is birational onto its image for all $m \geq 14$, which is optimal.
\end{abstract}
\maketitle

\pagestyle{myheadings}
\markboth{\hfill Y. Hu, J. Yan\hfill}{\hfill On the pluricanonical map of projective 3-folds of general type with $P_3 \geq 2$ \hfill}
\numberwithin{equation}{section}

\section{\bf Introduction}

One of the most fundamental problems in algebraic geometry is to classify algebraic varieties. Concerning this problem, one very effective approach is to study the pluricanonical maps.  A fundamental theorem established by Hacon-McKernan \cite{HM}, Takayama \cite{Takayama} and Tsuji \cite{Tsuji} guarantees the existence of optimal constants $r_n$ for any positive integer $n$ and for any smooth projective $n$-fold of general type, such that the $m$-canonical map $\varphi_m$ is birational onto its image for all $m \geq r_n$. Determining the exact values of $r_n$ is an important open question. It is classically known that $r_1=3$. For  surfaces of general type,  Bombieri \cite{Bom} proved that $r_2=5$. For $n=3$, it has been shown by Iano-Fletcher \cite{IF}, Chen-Chen \cite{EXP1, EXP2, EXP3} and Chen \cite{Chen18} that $27 \leq r_3 \leq 57$.

The pluricanonical map of $3$-folds with extra plurigenera conditions is expected to have better behavior. For instance, for  $3$-folds of general type with $p_g \geq 2$, Chen \cite{Chen03} proved that $\varphi_m$ is birational onto its image for all $m \geq 8$; for $3$-folds of general type with $P_2 \geq 2$, Chen-Chen \cite{EXP3} proved that $\varphi_m$ is birational onto its image for all $m \geq 11$. Both results are known to be optimal.

In this paper, we focus on nonsingular $3$-folds of general type with $P_3 \geq 2$ (see \cite[Remark 1.10]{EXP3} for the discussion on this topic). Our main result is the following theorem.
\begin{thm}\label{main theorem}
Let $V$ be a nonsingular projective $3$-fold of general type with $P_3(V) \geq 2$. Then:
\begin{enumerate}
    \item $\varphi_m$ is birational onto its image for all $m \geq 14$;
    \item if $\varphi_{13}$ is not birational, then $P_3(V)=2$, $q(V)=0$ and $|3K_V|$ is composed with a  pencil of $(1,2)$-surfaces.
\end{enumerate}
\end{thm}
The result is optimal due to the following example.
\begin{exmp}
A general hypersurface $X_{28}$ of degree $28$ in a weighted projective space $\mathbb{P}(1,3,4,5,14)$ (see \cite[p. 151]{IF}) is a minimal $3$-fold of general type with $P_3=2$ and its $13$-canonical map is non-birational.
\end{exmp}

In the study of explicit birational geometry of $3$-folds of general type, a significant challenge arises when the fibration induced by $|3K_V|$ has very small birational invariants. Our contribution  is to study a special birational model constructed in \cite{CHJ}. Using this new model, we are able to explicitly extract information about the singularities of the minimal model. By combining this with a detailed analysis of the pluricanonical maps, we are able to prove the main results stated in Theorem \ref{main theorem}.

Throughout this paper, all varieties are  defined over an algebraically closed field $k$ of characteristic $0$. We will frequently use the following symbols:
\begin{itemize}
\item[$\diamond$] `$\sim$' denotes linear equivalence;
\item[$\diamond$] `$\sim_{\mathbb{Q}}$' denotes $\mathbb{Q}$-linear equivalence;
\item[$\diamond$] `$\equiv$' denotes numerical equivalence;
\item[$\diamond$] `$|A| \succcurlyeq |B|$' or, equivalently, `$|B| \preccurlyeq |A|$'  means $|A| \supseteq |B|+$ fixed effective divisors.
\end{itemize}

\section{\bf Preliminaries}\label{pre}

\subsection{Varieties and divisors}
Let $Z$ be a normal projective variety of dimension $d$. The {\it geometric genus} $p_g(Z)$ and the {\it $m$-th plurigenus} $P_m(Z)$ of $Z$ are defined as
$$
p_g(Z):= h^0(Z, K_Z),\ P_m(Z):=h^0(Z, mK_Z).
$$
The {\it canonical volume} of $Z$ is defined as
$$
\mathrm{vol}(Z) := \limsup\limits_{n \to \infty} \frac{h^0(Z, nK_Z)}{n^d/d!}.
$$
We say that $Z$ is \emph{minimal}, if $Z$ has at worst $\bQ$-factorial terminal singularities and $K_Z$ is nef. If $Z$ is birational to a smooth projective variety $Z'$ with $\mathrm{vol}(Z') > 0$, we say that $Z$ is of general type.

If $Z$ has at worst rational singularities and $q(Z):=h^1(Z,\mathcal{O}_Z)>0$, then $Z$ is said to be \emph{irregular}. Otherwise, $Z$ is called \emph{regular}.

For two integers $a>0$ and $b\geq 0$, an {\it $(a, b)$-surface} is a smooth projective surface $S$ of general type with $\mathrm{vol}(S)=a$ and $p_g(S)=b$.

\subsection{Birationality principle}
Let $Z$ be a normal projective variety on which $D$ is a $\bQ$-Cartier Weil divisor with $h^0(Z, D)\geq 2$. We may define {\it a generic irreducible element of $|D|$} in the following way. Write
$$
|D|= \text{Mov}|D|+\text{Fix}|D|,
$$
where $\text{Mov}|D|$ and $\text{Fix}|D|$ denote the moving part and the fixed part of $|D|$ respectively. Consider the rational map $\varphi_{|D|}$ induced by the linear system $|D|$.  {\it $|D|$ is said to be composed with a pencil} if $\dim \overline{\varphi_{|D|}(Z)}=1$; otherwise, {\it $|D|$ is not composed with a pencil}. {\it A generic irreducible element of $|D|$} is defined to be an irreducible component of a general member in $\text{Mov}|D|$ if $|D|$ is composed with a pencil or,  otherwise,  a general member of $\text{Mov}|D|$.

\begin{defn}
 We say that $|D|$ can {\it distinguish different generic irreducible elements $M_1$ and $M_2$ of a linear system $|M|$} if neither $M_1$ nor $M_2$ is contained in $\text{Bs}|D|$,  and if $\overline{\varphi_{|D|}(M_1)} \nsubseteq \overline{\varphi_{|D|}(M_2)}, \overline{\varphi_{|D|}(M_2)} \nsubseteq \overline{\varphi_{|D|}(M_1)}$.
\end{defn}

We will use the following type of birationality principle.

\begin{thm} (cf. \cite[\S 2.7]{EXP2}) \label{thm: bir. prin.}
Let $Z$ be a nonsingular projective variety, $A$ and $B$ be two divisors on $Z$ with $|A|$ being a base point free linear system. Take the Stein factorization of $\varphi_{|A|}$: $Z \overset{h} \longrightarrow W \longrightarrow \bP^{h^0(Z,A)-1}$
where $h$ is a fibration onto a normal variety $W$. Then the rational map $\varphi_{|B+A|}$ is birational onto its image if one of the following conditions is satisfied:
\begin{itemize}
\item [(i)] $\dim\varphi_{|A|}(Z)\geq 2$, $|B|\neq \emptyset$ and $\varphi_{|B+A|}|_D$ is birational for a
general member $D$ of $|A|$.

\item [(ii)] $\dim\varphi_{|A|}(Z)=1$, $\varphi_{|B+A|}$ can distinguish different general fibers of $h$ and $\varphi_{|B+A|}|_F$ is birational for a general fiber $F$ of $h$.
\end{itemize}
\end{thm}

\subsection{Setting}\label{setup}
Let $X$ be a minimal $3$-fold of general type with $P_3(X)\geq 2$. Take a general $2$-dimensional linear subspace $V\subseteq H^0(X, 3K_X)$ and consider the corresponding sublinear system $\Lambda\subseteq |3K_X|$.
We can consider the rational map defined by $\Lambda$, say
$\Phi_{\Lambda}: X {\dashrightarrow} \bP^1$, which is
not necessarily well-defined everywhere. By Hironaka's big
theorem, we can take successive blow-ups $\pi: X'\to X$ such
that:
\begin{enumerate}
\item $X'$ is smooth and projective;
\item corresponding to $\Lambda$, the sublinear system
$\Lambda'\subseteq |\rounddown{\pi^*(3K_X)}|$ has free movable part $\Mov(\Lambda')$ and, consequently,
the rational map $\gamma=\Phi_{\Lambda}\circ \pi$ is a morphism;
\item for a fixed general member $D$ in $\Lambda$, the union of supports of $\pi^*(D)$ and $\pi$-exceptional divisors is of simple normal crossing.
\end{enumerate}
Let $X'\overset{f}\longrightarrow \Gamma\overset{s}\longrightarrow \bP^1$
be the Stein factorization of $\gamma$. There is a commutative
diagram:
\[\xymatrix@=4em{
X'\ar[d]_\pi \ar[dr]^{\gamma} \ar[r]^f& \Gamma\ar[d]^s\\
X \ar@{-->}[r]^{\Phi_{\Lambda}} & \bP^1}
\]

Let $F$ be a general fiber of $f$. By Bertini's theorem, $F$ is a smooth surface of general type. We may write
\begin{align}\label{eq: resolution}
K_{X'}=\pi^*(K_X)+E_\pi,\ \pi^*(3K_X)=F+R,
\end{align}
 where $E_\pi$ is an effective $\pi$-exceptional $\bQ$-divisor and $R$ is an effective $\bQ$-divisor. Set $F_X=\pi_*(F)$. Then we may write
 \begin{align}\label{eq: fixed part 1}
     3K_X=F_X+Z_X,
 \end{align}
 where $Z_X$ is an effective divisor.
\subsection{Plurigenera of threefolds of general type}
Applying the proof of \cite[Lemma 3.2]{EXP2}, we obtain the following lemma.
\begin{lem}\label{lem: pluricanonical genera}
    Let $X$ be a minimal $3$-fold of general type. Then
    the following hold:
    \begin{enumerate}
    \item for all $n\geq 2$, we have
    $$
    P_{n+2}(X)\geq P_n(X)+P_2(X)+(n^2+n)K_X^3-\chi(\OX);
    $$
        \item if $\chi(\OX)\leq 2$ and $P_3(X)\geq 2$, we have $P_5(X)\geq 1$;
        \item if $\chi(\OX)\leq 1$ and $P_3(X)\geq 2$, we have $P_n(X)\geq 1$ for all $n\geq 8$.
    \end{enumerate}
\end{lem}
\begin{proof}
  (1) follows by the proof of \cite[Lemma 3.2]{EXP2}. For (2), by (1) and our assumption, we have
   \begin{align*}
    P_5(X)&\geq P_3(X)+P_2(X)+12K_X^3-\chi(\OX)\\
    &\geq P_2(X)+12K_X^3\\
    &>0.
   \end{align*}

Now suppose that $\chi(\OX)\leq 1$. By (1), for all $k\geq 6$, we have
\begin{align*}
    P_{k+2}(X)&\geq P_k(X)+P_2(X)+(k^2+k)K_X^3-\chi(\OX)\\
    &\geq \frac{k^2+k}{36}-1\\
    &>0,
\end{align*}
where the second inequality follows by \cite[Proposition 4.3]{EXP3} and the last inequality holds by $k\geq 6$. It follows that we have $P_n(X)\geq 1$ for all $n\geq 8$. (3) is proved. The proof is completed.

\end{proof}

\section{\bf Birationality of pluricanonical maps}
According to the minimal model program (see \cite{BCHM}), a smooth projective variety of general type is always birational to a {\it minimal model}. To study the birationality of pluricanonical maps of varieties of general type, it suffices to study the pluricanonical maps of their minimal models.

Let $X$ be a minimal  $3$-fold of general type with $P_3(X) \geq 2$.  The following lemma follows by \cite[Proposition 2.9]{CH}, \cite[Theorem 1.1]{CCJ} and  \cite[Theorem 1.3]{CCCJ}.

\begin{lem}\label{lem: irregular}
    Let $X$ be an irregular minimal $3$-fold of general type. Then $\varphi_m$ is birational for all $m\geq 5$.
\end{lem}

\begin{proof}
    By \cite[Proposition 2.9]{CH}, $\varphi_m$ is birational for all $m\geq 7$. The birationality of $\varphi_6$ follows by \cite[Theorem 1.1]{CCJ}. $\varphi_5$ is birational by \cite[Theorem 1.3]{CCCJ}.
\end{proof}
Thus we only need to consider the case when $X$ is regular. From now on, we always assume that $q(X)=0$. Take $\Lambda\subseteq|3K_X|$ and $\pi: X'\to X$ as in \S \ref{setup}. Keep the same notations as there. Recall that we have the fibration $f: X'\to \Gamma$ with a general fiber $F$. Since $q(X)=0$, we have $\Gamma\cong \bP^1$.

\begin{lem}\label{lem: disting. g.i.e s}
The linear system $|mK_{X'}|$ distinguishes different general fibers of $f$ for all $m \geq 12$.
\end{lem}
\begin{proof}
When $p_g(F)>0$, by \cite[Proposition 2.15]{EXP2}, we have $P_n(X) \geq 2$ for all $n \geq 9$. So $mK_{X'} \geq F$ for all $m \geq 9+3=12$. It follows that $|mK_{X'}|$ distinguishes different general fibers of $f$.

When $p_g(F)=0$, by \cite[Lemma 2.32]{EXP2}, we have $\chi(\mathcal{O}_X) \leq 1$. By Lemma \ref{lem: pluricanonical genera} (3), we have $P_n(X)\geq 1$ for all $n\geq 8$.
So for all $m \geq 8+3=11$,  we have $mK_{X'} \geq F$, which implies that $|mK_{X'}|$ distinguishes different general fibers of $f$. The proof is completed.

\end{proof}

So by Theorem \ref{thm: bir. prin.}, in order to consider the birationality of $\varphi_m$, we may just consider the birationality of the map induced by $|mK_{X'}||_F$. Denote by $\sigma: F\to F_0$ the contraction onto its minimal model. By \cite[Corollary 2.3]{Noether} (take $\lambda=3$, $D=K_X$ and $S=F$), there is an effective $\bQ$-divisor $H$ such that
\begin{align}\label{eq: restriction}
    \pi^*(K_X)|_F\simQ \frac{1}{4}\sigma^*(K_{F_0})+H.
\end{align}

Set $L_m=(m-4)\pi^*(K_X)$. It is clear that $L_m$ is a nef and big divisor whose fractional part is simple normal crossing for any $m\geq 5$. For all $m\geq 5$, we have
\begin{align}\label{eq: restriction of linear system} |mK_{X'}||_F\lsgeq|K_{X'}+\roundup{L_m}+F||_F
  \lsgeq |K_F+\roundup{L_m}|_F|
  \lsgeq |K_F+\roundup{L_m|_F}|,
\end{align}
where the first inequality follows by \eqref{eq: resolution} and the second inequality follows by Kawamata-Viehweg vanishing theorem.
\begin{lem}\label{lem: birationality on surface}
    Let $X$ be a regular minimal $3$-fold of general type with $P_3(X)\geq 2$. Then $|K_F+\roundup{L_m|_F}|$ gives a birational map if one of the following holds:
    \begin{enumerate}
        \item $K_{F_0}^2\geq 2$ and $m\geq 13$;
        \item $F$ is a $(1,1)$-surface and $m\geq 14$.
    \end{enumerate}
\end{lem}
\begin{proof}
Note that we have $m\geq 13$ in either case.
Let $C_x$ be any irreducible curve passing through a very general point $x\in F$. We have
$$(L_m|_F\cdot C_x) \geq \frac{m-4}{4}(\sigma^*(K_{F_0}) \cdot C_x) \geq \frac{m-4}{2}>4,$$
where the first inequality follows by \eqref{eq: restriction}, the second inequality holds by \cite[Lemma 2.5]{EXP3} and the last inequality holds since $m\geq 13$.

If $K_{F_0}^2\geq 2$ and $m\geq 13$, we have
$$(L_m|_F)^2 = ((m-4)\pi^*(K_X)|_F)^2 \geq \frac{1}{16}(m-4)^2 K_{F_0}^2 >8,$$
where the first inequality follows by \eqref{eq: restriction} and the second inequality holds since $m\geq 13$. By \cite[Lemma 2.3]{EXP3}, $|K_F+\roundup{L_m|_F}|$ gives a birational map.

If $F$ is a $(1,1)$-surface and $m\geq 14$, take $|G|=|2\sigma^*(K_{F_0})|$, then $|G|$ is base point free by surface theory. Take a general member $C$ of $|G|$, then by \cite[Table A4]{EXP3}, we have
\begin{align*}
    (\pi^*(K_X)|_F\cdot \sigma^*(K_{F_0}))=\frac{1}{2} (\pi^*(K_X)|_F\cdot C)\geq \frac{1}{3}.
\end{align*}
It follows that
$$
(\pi^*(K_X)|_F)^2\geq \frac{1}{4}(\pi^*(K_X)|_F\cdot \sigma^*(K_{F_0}))\geq \frac{1}{12},
$$
where the first inequality follows by \eqref{eq: restriction}. Since $m\geq 14$, we have
$$
(L_m|_F)^2=(m-4)^2(\pi^*(K_X)|_F)^2 \geq\frac{100}{12}>8.
$$
By \cite[Lemma 2.3]{EXP3}, $|K_F+\roundup{L_m|_F}|$ gives a birational map. The proof is completed.

\end{proof}

\subsection{The case when $F$ is not a $(1,2)$-surface}
In this section, we consider the case when $F$ is not a $(1,2)$-surface.

\begin{prop}\label{prop: mK vol(F)>=2}
    Let $X$ be a regular minimal $3$-fold of general type with $P_3(X)\geq 2$. Suppose that $K_{F_0}^2\geq 2$ for a general fiber $F$ of $f$. Then $\varphi_m$ is birational for all $m\geq 13$.
\end{prop}
\begin{proof}
 By Lemma \ref{lem: birationality on surface}, $|K_F+\roundup{L_m|_F}|$ gives a birational map for all $m\geq 13$, and consequently, by \eqref{eq: restriction of linear system}, $|mK_{X'}||_F$ gives a birational map. So by Theorem \ref{thm: bir. prin.} and Lemma \ref{lem: disting. g.i.e s}, we deduce that $\varphi_m$ is birational for all $m\geq 13$. The proof is completed.

\end{proof}
\begin{prop}\label{prop: mK (1,0)}
    Let $X$ be a regular minimal $3$-fold of general type with $P_3(X)\geq 2$. Suppose that $F$ is a $(1,0)$-surface. Then $\varphi_m$ is birational for all $m\geq 13$.
\end{prop}

\begin{proof}
By \cite[Lemma 2.32]{EXP2}, we have $\chi(\OX)\leq 1$.

{\bf Case 1.} $P_2(X)>0$.

By Lemma \ref{lem: pluricanonical genera} (1), for all $n\geq 2$, we have
$$
P_{n+2}(X)> P_n(X)+P_2(X)-1.
$$
It follows that $P_5(X)\geq 3$ and $P_n(X)>0$ for all $n\geq 2$. Denote by $|M_5|$ the movable part of $|5K_{X'}|$. After taking some birational modification of $X'$, we may assume that $|M_5|$ is base point free.

If $|M_5|$ and $f$ are not composed with the same pencil, then $M_5|_F$ is non-zero and movable. Thus we have $(\sigma^*(K_{F_0})\cdot M_5|_F)\geq 2$ by \cite[Lemma 2.4]{EXP3}. We deduce that
$$
(\pi^*(K_X)|_F)^2\geq \frac{1}{4}(\sigma^*(K_{F_0})\cdot\pi^*(K_X)|_F)\geq \frac{1}{20}(\sigma^*(K_{F_0})\cdot M_5|_F)\geq \frac{1}{10},
$$
where the first inequality follows by \eqref{eq: restriction}. For any $m\geq 13$ and any irreducible curve $C_x$ passing through a very general point $x\in F$, we have
$$
(L_m|_F\cdot C_x)\geq \frac{m-4}{4}(\sigma^*(K_{F_0})\cdot C_x)\geq \frac{m-4}{2}>4,
$$
and
$$
(L_m|_F)^2=(m-4)^2(\pi^*(K_X)|_F)^2\geq \frac{81}{10}>8.
$$
By \cite[Lemma 2.3]{EXP3}, $|K_F+\roundup{L_m|_F}|$ gives a birational map. So by \eqref{eq: restriction of linear system}, $|mK_{X'}||_F$ gives a birational map for all $m\geq 13$. Therefore, by Theorem \ref{thm: bir. prin.} and Lemma \ref{lem: disting. g.i.e s}, we conclude that $\varphi_m$ is birational for all $m\geq 13$.

If $|M_5|$ and $f$ are composed with the same pencil, we have $|M_5|\lsgeq |2F|$. By \cite[Theorem 2.2]{Noether}, the natural restriction map
$$
H^0(X', 3(K_{X'}+F))\to H^0(F, 3K_F)
$$
is surjective. By \cite[Theorem 3]{Miyaoka}, the $3$-canonical map of $F$ is birational. Since $P_3(X)\geq 2$ and $F$ is general, we conclude that  $|3(K_{X'}+F)|$ gives a birational map.  Note that we have
$$
|11K_{X'}|\lsgeq |3K_{X'}+F+M_5|\lsgeq |3(K_{X'}+F)|.
$$
It follows that $\varphi_{11}$ is birational. Since $P_n(X)>0$ for all $n\geq 2$, we conclude that $\varphi_m$ is birational for all $m\geq 13$.

{\bf Case 2.} $P_2(X)=0$.

By \cite[equation (3.6)]{EXP1}, we have
\begin{align*}
    0\leq n_{1,2}^0=5\chi(\OX)+6P_2(X)-4P_3(X)+P_4(X)\leq 5-4P_3(X)+P_4(X),
\end{align*}
which implies that $P_4(X)\geq 4P_3(X)-5\geq 3$.  Denote by $|M_4|$ the movable part of $|4K_{X'}|$. After taking some further birational modification of $X'$, we may assume that $|M_4|$ is base point free.

If $|M_4|$ and $f$ are not composed with the same pencil,  $M_4|_F$ is non-zero and movable on $F$. By the same argument as in {\bf Case 1}, we conclude that $\varphi_m$ is birational for all $m\geq 13$.

If $|M_4|$ and $f$ are composed with the same pencil, we have $|M_4|\lsgeq |2F|$. By the same argument as in {\bf Case 1}, $|3(K_{X'}+F)|$ gives a birational map. Note that we have
$$
|10K_{X'}|\lsgeq |3K_{X'}+M_4+F|\lsgeq |3(K_{X'}+F)|.
$$
It follows that $\varphi_{10}$ is birational. By Lemma \ref{lem: pluricanonical genera}, we have $P_5(X)\geq 1$ and $P_n(X)\geq 1$ for all $n\geq 8$. Since $P_3(X)\geq 2$ and $P_4(X)\geq 3$, it is clear that we have
$P_6(X)>0$ and $P_7(X)>0$.
We deduce that $P_n(X)\geq 1$ for all $n\geq 3$. So $|mK_{X'}|\lsgeq |10K_{X'}|$ for all $m\geq 13$. Thus $\varphi_m$ is birational for all $m\geq 13$. The proof is completed.

\end{proof}
\begin{prop}\label{prop: mK (1,1)}
    Let $X$ be a regular minimal $3$-fold of general type with $P_3(X)\geq 2$.  Suppose that $F$ is a $(1,1)$-surface. Then $\varphi_m$ is birational for all $m\geq 13$.
\end{prop}
\begin{proof}
By Lemma \ref{lem: birationality on surface}, $|K_F+\roundup{L_m|_F}|$ gives a birational map for all $m\geq 14$. By \eqref{eq: restriction of linear system}, $|mK_{X'}||_F$ gives a birational map for all $m\geq 14$. We conclude that $\varphi_m$ is birational for all $m\geq 14$ by Theorem \ref{thm: bir. prin.} and Lemma \ref{lem: disting. g.i.e s}.

We are left to treat the case when $m=13$. Take $|G|=|2\sigma^*(K_{F_0})|$. Then $|G|$ gives a generic finite morphism by surface theory. Moreover, a general member $C\in |G|$ is a smooth non-hyperelliptic curve.

By \cite[Theorem 2.1]{Kollar}, $R^1f_*\omega_{X'}$ is torsion-free of rank $q(F)$. By \cite[Theorem 11]{Bom}, we have $q(F)=0$. It follows that we have $R^1f_*\omega_{X'}=0$. Since $f_*\omega_{X'/{\bP^1}}$ is a nef line bundle, we have $h^1(\bP^1, f_*\omega_{X'})\leq 1$. We deduce that
    $$
    h^2(\OX)=h^1(\omega_{X'})=h^0(\bP^1, R^1f_*\omega_{X'})+h^1(\bP^1, f_*\omega_{X'})\leq 1.
    $$
    So
    $$
    \chi(\OX)=1+h^2(\OX)-p_g(X)-q(X)\leq 2.
    $$
By Lemma \ref{lem: pluricanonical genera} (2), we have $P_5(X)\geq 1$. It follows that we have
$$
|13K_{X'}||_F\lsgeq |8K_{X'}||_F\lsgeq |2(K_{X'}+F)||_F\lsgeq |G|,
$$
where the last inequality follows by \cite[Theorem 2.2]{Noether}. So $|13K_{X'}||_F$ distinguishes different general members of $|G|$. By Theorem \ref{thm: bir. prin.} and Lemma \ref{lem: disting. g.i.e s}, it suffices to show that $(|13K_{X'}||_F)|_C$ gives a birational map. By \eqref{eq: restriction of linear system}, we have
$$
|13K_{X'}||_F\lsgeq |K_F+\roundup{9\pi^*(K_X)|_F}|.
$$
By Kawamata-Viehweg vanishing theorem and \eqref{eq: restriction}, we have
\begin{align*}
 |K_F+\roundup{9\pi^*(K_X)|_F}||_C &\lsgeq |K_F+\roundup{9\pi^*(K_X)|_F-8H}||_C\\
 &\lsgeq |K_C+\roundup{9\pi^*(K_X)|_F-8H-C}|_C|\\
 &=|K_C+D|,
\end{align*}
where $D=(\roundup{9\pi^*(K_X)|_F-8H-C})|_C$.
Note that we have
$$
9\pi^*(K_X)|_F-8H-C\equiv \pi^*(K_X)|_F.
$$
It follows that $(\sigma^*(K_{F_0})\cdot (\roundup{9\pi^*(K_X)|_F-8H-C}))\geq 1$. Thus we have $\deg(D)\geq 2$. Since $C$ is non-hyperelliptic, $|K_C+D|$ gives a birational map by \cite[Fact 2.3]{Chen03}. Thus $\varphi_{13}$ is birational. The proof is completed.

\end{proof}

\subsection{The case when $F$ is a $(1,2)$-surface}
In this section, we always assume that a general fiber $F$ of $f: X'\to \bP^1$ is a $(1,2)$-surface. After taking some birational modification of $X'$, we may assume that $\Mov|K_F|$ is base point free. Let $C\in\Mov|K_F|$ be a general member. It is known that $C$ is a smooth curve of genus $2$. Set $\xi=(\pi^*(K_X)\cdot C)$. By \cite[Table A3]{EXP3}, we have
\begin{align}\label{eq: xi}
    \xi\geq \frac{1}{3}.
\end{align}
By \eqref{eq: restriction}, there exists an effective $\bQ$-divisor $E_F$ such that
\begin{align}\label{eq: restriction 2}
    \pi^*(K_X)|_F\simQ \frac{1}{4}C+E_F.
\end{align}
It follows easily that $(\pi^*(K_X)|_F)^2\geq \frac{1}{12}$.

\begin{lem}\label{lem: separate genus 2 curve}
 $|mK_{X'}||_F$ distinguishes different general members of $|C|$ for all $m \geq 13$.
\end{lem}
\begin{proof}
   Let $m\geq 13$ be an integer. As $p_g(F)>0$, by \cite[Proposition 2.15]{EXP2}, we have $P_n(X) \geq 2$ for all $n \geq 9$. Together with \eqref{eq: restriction of linear system}, it follows that 
   $$
   |mK_{X'}||_F\lsgeq |K_F+\roundup{L_m|_F}|\lsgeq 
   \Mov|K_F|\lsgeq 
   |C|.
   $$
   So $|mK_{X'}||_F$ distinguishes different general members of $|C|$.

\end{proof}
\begin{prop}\label{prop: birationality (1,2) surface 1}
Suppose that $\xi > \frac{1}{3}$. Then $\varphi_m$ is birational for all $m\geq 14$.
\end{prop}
\begin{proof}
Suppose that $m\geq 14$. By \eqref{eq: restriction of linear system}, we have
$$
|mK_{X'}||_F\lsgeq |K_F+\roundup{L_m|_F}|.
$$
By \eqref{eq: restriction 2},  it is clear that
$$
L_m|_F-4E_F-C\equiv (m-8)\pi^*(K_X)|_F
$$
is nef and big. By Kawamata-Viehweg vanishing theorem, we have
$$
|K_F+\roundup{L_m|_F}||_C\lsgeq |K_F+\roundup{L_m|_F-4E_F}||_C\lsgeq |K_C+\roundup{L_m|_C-4{E_F}|_C}|.
$$
Note that we have
$$
\deg (L_m|_C-4{E_F}|_C)=(m-8)(\pi^*(K_X)\cdot C)=(m-8)\xi>2,
$$
where the last inequality holds since $m\geq 14$ and $\xi>\frac{1}{3}$. Thus $|K_C+\roundup{L_m|_C-4{E_F}|_C}|$ gives a birational map. By Theorem \ref{thm: bir. prin.}, Lemma \ref{lem: disting. g.i.e s} and Lemma \ref{lem: separate genus 2 curve}, $\varphi_m$ is birational for all $m\geq 14$. The proof is completed.

\end{proof}
\begin{prop}\label{prop: birationality (1,2) surface 2}
Suppose that $\xi=\frac{1}{3}$ and $(\pi^*(K_X)|_F)^2>\frac{1}{12}$, then $\varphi_m$ is birational for all $m\geq 14$.
\end{prop}

\begin{proof}
Consider the Zariski decomposition of the following $\mathbb{Q}$-divisor:
$$6\pi^*(K_X)|_F+4E_F=(6\pi^*(K_X)|_F+N^+)+N^-,$$
where
\begin{enumerate}
    \item[(z1)]  $N^+$ and $N^-$ are effective $\mathbb{Q}$-divisors such that $N^++N^-=4E_F$;
    \item[(z2)] the $\mathbb{Q}$-divisor $6\pi^*(K_X)|_F+N^+$ is nef;
    \item[(z3)] $((6\pi^*(K_X)|_F+N^+) \cdot N^-)=0$.
\end{enumerate}

{\bf Step 1.} In this step, we prove that $(\pi^*(K_X)|_F\cdot N^+)>0$.

Note that we have
\begin{align*}
\frac{1}{12}<(\pi^*(K_X)|_F)^2=& \frac{1}{4}(\pi^*(K_X)|_F \cdot C)+(\pi^*(K_X)|_F \cdot E_F)\\
=& \frac{1}{12}+(\pi^*(K_X)|_F \cdot E_F),
\end{align*}
where the first equality follows by \eqref{eq: restriction 2}. Thus $(\pi^*(K_X)|_F\cdot E_F)>0$. We deduce that $N^+\neq 0$ by (z1) and (z3). By (z2), it is clear that $6\pi^*(K_X)|_F+N^+$ is nef and big. We claim that $(\pi^*(K_X)|_F\cdot N^+)>0$. Otherwise, by \cite[Lemma 2.2(2)]{CHJ}, either $6\pi^*(K_X)|_F=0$ or $N^+=0$. It is a contradiction. Thus we have $(\pi^*(K_X)|_F\cdot N^+)>0$.

\textbf{Step 2.} In this step, we prove that $(N^+ \cdot C)>0$.

Since $C$ is nef, we see that $(N^+ \cdot C) \geq 0$. Assume to the contrary that $(N^+ \cdot C)=0$, then we have $(N^+)^2 \leq 0$ by Hodge index theorem. We deduce that
$$
(N^+\cdot(N^++N^-))=4(N^+\cdot E_F)=4(N^+\cdot \pi^*(K_X)|_F)>0,
$$
where the first equality holds by (z1), the second equality follows by \eqref{eq: restriction 2} and the last inequality holds by {\bf Step 1}. Since $(N^+)^2\leq 0$, we have
$(N^+\cdot N^-)>0$. By (z3), we have $$
(\pi^*(K_X)|_F\cdot N^-)=-\frac{1}{6}(N^+\cdot N^-)<0,
$$
which is a contradiction. Thus we have $(N^+\cdot C)>0$.

{\bf Step 3.} In this step, we prove that $\varphi_{m}$ is birational for all $m\geq 14$.

Let $m\geq 14$ be an integer. By \eqref{eq: restriction of linear system}, we have
$$
|mK_{X'}||_F\lsgeq |K_F+\roundup{L_m|_F}|.
$$
Noting that
\begin{align*}
L_m|_F = & (m-4)\pi^*(K_X)|_F\\
                      \equiv & (m-8)\pi^*(K_X)|_F+C+4E_F\\
                      \equiv & ((m-8)\pi^*(K_X)|_F+N^+)+C+N^-,
\end{align*}
and
$$
(m-8)\pi^*(K_X)|_F+N^+=(m-14)\pi^*(K_X)|_F+6\pi^*(K_X)|_F+N^+
$$
is nef and big, the vanishing theorem gives that
$$
|K_F+\roundup{L_m|_F}||_C\lsgeq|K_F+\roundup{L_m|_F-N^-}||_C \lsgeq |K_C+\roundup{D^+}|,$$
where
$D^+ = (L_m|_F-N^-)|_C$.
We deduce that
\begin{align*}
    \deg(D^+)=(m-8)(\pi^*(K_X)|_F\cdot C)+(N^+\cdot C)\geq 2+(N^+\cdot C)>2,
\end{align*}
where the last inequality holds by {\bf Step 2}. So $|K_C+\roundup{D^+}|$ gives a birational map. By Theorem \ref{thm: bir. prin.}, Lemma \ref{lem: disting. g.i.e s} and Lemma \ref{lem: separate genus 2 curve}, we conclude that $\varphi_m$ is birational for all $m\geq 14$. The proof is completed.

\end{proof}
In the following part, we will assume that $\xi = \frac{1}{3}$ and $(\pi^*(K_X)|_F)^2=\frac{1}{12}$. By \eqref{eq: restriction}, \eqref{eq: restriction 2} and our assumption, it is clear that
\begin{align}
    (\pi^*(K_X)|_F\cdot \sigma^*(K_{F_0}))=\frac{1}{3}.
\end{align}

We will adopt the idea in   \cite{CHJ}. Applying \cite[Proposition 3.1]{CHJ} to $\Lambda\subseteq |3K_X|$, we get a projective birational morphism $\mu: W\to X$ with a surjective fibration $g: W\to\Gamma$ such that $W$ is $\mathbb{Q}$-factorial terminal and
\begin{align}\label{eq: G}
    \mu^*(K_X+F_X)-K_W-F_W=G
\end{align}
is an effective $\mu$-exceptional divisor, where $F_W$ is a general fiber of $g$ satisfying $F_X={\mu}_*(F_W)$. Moreover, $K_W+F_W$ is $\mu$-nef. Since $F$ is birational to $F_W$, the minimal model of $F_W$ is $F_0$. Denote by $\sigma_W: F_W\to F_0$ the contraction. Note that $G$ is independent of $F_W$ by the negativity lemma \cite[Lemma 3.39]{KM}. We may write
\begin{align}\label{eq: fixed part 2}
  K_W=\mu^*(K_X)+E_\mu,
\end{align}
where $E_\mu$ is an effective $\mu$-exceptional divisor. By \eqref{eq: fixed part 1}, \eqref{eq: G} and adjunction formula, we have
\begin{align}\label{eq: mu^*K}
    4\mu^*(K_X)|_{F_W}=K_{F_W}+G|_{F_W}+\mu^*(Z_X)|_{F_W}.
\end{align}

\begin{lem}\label{lem: -3-curve}
    There exists a unique $\mu$-exceptional prime divisor $E_0$ such that $(\sigma_W^*(K_{F_0})\cdot E_0|_{F_W})>0$ and $E_0\subseteq \mathrm{Supp}(G+\mu^*(Z_X))$. Moreover, the following hold:
    \begin{enumerate}
    \item $\mu(E_0)$ is a point;
     \item $E_0|_{F_W}$ is an integral $(-3)$-curve;
        \item $F_W$ is minimal;
        \item $G|_{F_W}=\frac{1}{3}E_0|_{F_W}$ and $\mu^*(Z_X)|_{F_W}=0$;
       \item $Z_X=0$.
    \end{enumerate}
\end{lem}
\begin{proof}
Note that we have
    $$
    (\mu^*(K_X)|_{F_W}\cdot \sigma_W^*(K_{F_0}))=(\pi^*(K_X)|_F\cdot\sigma^*(K_{F_0}))=\frac{1}{3}.
    $$
By \eqref{eq: mu^*K}, we have
$$
(\sigma_W^*(K_{F_0})\cdot (G|_{F_W}+\mu^*(Z_X)|_{F_W}))=4(\sigma_W^*(K_{F_0})\cdot\mu^*(K_X)|_{F_W})-1=\frac{1}{3}.
$$
Thus there is an integral curve $A\subseteq\mathrm{Supp}(G|_{F_W}+\mu^*(Z_X)|_{F_W})$ such that $(\sigma_W^*(K_{F_0})\cdot A)>0$. Denote by $\lambda$ the coefficient of $A$ in $G|_{F_W}+\mu^*(Z_X)|_{F_W}$. It is clear that $\lambda\leq \frac{1}{3}$, which implies that $A$ is not contained in $\mu^{-1}_*(Z_X)$. Since $F_W$ is general, there is a $\mu$-exceptional prime divisor $E_0\subseteq \mathrm{Supp}(G+\mu^*(Z_X))$ such that $A\subseteq E_0|_{F_W}$. Thus the coefficient of $E_0$ in $G+\mu^*(Z_X)$ is $\lambda\leq \frac{1}{3}$.

We claim that $\mu(E_0)$ is a point. Otherwise, $\mu(E_0)$ is a curve. Since $X$ has at worst isolated singularities, $X$ is smooth at the general point of $\mu(E_0)$. It follows that the coefficient of $E_0$ in $G+\mu^*(Z_X)$ is a positive integer, which is a contradiction.  (1) is proved.

By (1), we have $(\mu^*(K_X)|_{F_W}\cdot A)=0$. By \eqref{eq: mu^*K}, we deduce that
$$
0\geq (K_{F_W}\cdot A)+\lambda A^2\geq (1-\lambda)(K_{F_W}\cdot A)-2\lambda\geq 1-3\lambda\geq 0,
$$
where the second inequality holds since $A^2\geq -2-(K_{F_W}\cdot A)$, the third inequality follows by $(K_{F_W}\cdot A)\geq (\sigma_W^*(K_{F_0})\cdot A)\geq 1$. Thus all the above inequalities become equalities. In particular, we have $(K_{F_W}\cdot A)=1$, $A^2=-3$ and $\lambda=\frac{1}{3}$. So $A$ is an integral $(-3)$-curve.

It is clear that $\sigma_W^*(K_{F_0})+\frac{1}{3}A$ is nef. An easy computation shows that
$$
(\sigma_W^*(K_{F_0})+\frac{1}{3}A)^2=16(\mu^*(K_X)|_{F_W})^2=\frac{4}{3}.
$$
By \cite[Lemma 2.2]{CHJ} and \eqref{eq: mu^*K}, we have
\begin{align*}
    4\mu^*(K_X)|_{F_W}=\sigma_W^*(K_{F_0})+\frac{1}{3}A.
\end{align*}
It follows easily that $E_0|_{F_W}=A$, $K_{F_W}=\sigma_W^*(K_{F_0})$. (2) and (3) follows. Moreover, we have
$$
G|_{F_W}+\mu^*(Z_X)|_{F_W}=\frac{1}{3}E_0|_{F_W}.
$$
Since $\mu(E_0)$ is a point,  we have $\mu^*(Z_X)|_{F_W}=0$. Thus we have $G|_{F_W}=\frac{1}{3}E_0|_{F_W}$. (4) is proved.

(5) follows by the same argument as in the proof of \cite[Lemma 3.5]{CHJ}. We omit the details. The proof is completed.

\end{proof}

Note that $(F_W, G|_{F_W})$ is klt and $Z_X=0$ by Lemma \ref{lem: -3-curve}. Applying the proof of \cite[Lemma 3.4]{CHJ} verbatim, we obtain the following lemma.
\begin{lem}\label{lem: singularity of 3K}
Keep the same notations as above. Then the following hold:
\begin{enumerate}
    \item $(X, F_X)$ is plt, $F_X$ is normal and klt, and
    $$
    K_{F_W}+\frac{1}{3}E_0|_{F_W}=(\mu|_{F_W})^*(K_{F_X}).
    $$
    \item For any non-Gorenstein terminal singularity $P\in X$, denote by $r_P$ the Cartier index of $K_X$ at $P$. If $P\in F_X$, then $r_P$ is equal to the Cartier index of $K_X|_{F_X}$ at $P$.
\end{enumerate}
\end{lem}

\begin{prop}\label{prop: index of X}
    Let $X$ be a regular minimal $3$-fold of general type with $P_3(X)\geq 2$. Suppose that $\xi=\frac{1}{3}$. Then we have $(\pi^*(K_X)|_F)^2>\frac{1}{12}$.
\end{prop}
\begin{proof}
    Suppose that the inequality does not hold. Then by \eqref{eq: restriction 2}, we have $(\pi^*(K_X)|_F)^2=\frac{1}{12}$. Note that $Z_X=0$ by Lemma \ref{lem: -3-curve} (5). It follows that $F_X\in |3K_X|$. Let $P\in X$ be any non-Gorenstein terminal singular point whose index $r_P$ of $K_X$ at $P$ is not $3$. It is clear that $P\in F_X$.

    {\bf Step 1.} We prove that $F_X$ has exactly one non-Gorenstein singularity which is of type $\frac{1}{3}(1,1)$.

By Lemma \ref{lem: -3-curve} and Lemma \ref{lem: singularity of 3K} (1), we have
$$
K_{F_W}+\frac{1}{3}E_0|_{F_W}=(\mu|_{F_W})^*(K_{F_X}).
$$
So outside the point $P_0=\mu|_{F_W}(E_0|_{F_W})$, $F_X$ has at worst Du Val singularities which are clearly Gorenstein.  We claim that $\mu|_{F_W}^{-1}(P_0)$ is just the  $(-3)$-curve $E_0|_{F_W}$. Otherwise, there is an integral curve $C'\subseteq \mu|_{F_W}^{-1}(P_0)$ such that $C'\neq E_0|_{F_W}$ and $(C'\cdot E_0|_{F_W})\geq 1$. We have
$$
(K_{F_W}\cdot C')=-\frac{1}{3}(E_0|_{F_W}\cdot C')<0,
$$
which is a contradiction by Lemma \ref{lem: -3-curve} (3). Thus $E_0|_{F_W}$ is contracted  to a singularity of type $\frac{1}{3}(1,1)$.

{\bf Step 2.} In this step, we will prove that $r_P\in \{ 2,3,4\}$.

By {\bf Step 1}, $P\in F_X$ is either a Gorenstein singularity or a singularity of type $\frac{1}{3}(1,1)$.

If $P$ is a Gorenstein singularity of $F_X$, then $4K_X|_{F_X}=K_{F_X}$ is Cartier at $P$. By Lemma \ref{lem: singularity of 3K} (2), $r_P$ is $2$ or $4$.
If $P\in F_X$ is a singularity of type $\frac{1}{3}(1,1)$, then $3K_X|_{F_X}$ is Cartier at $P$. By Lemma \ref{lem: singularity of 3K} (2), $r_P$ is $3$.

{\bf Step 3.} In this step, we will prove that $K_X^3\geq \frac{1}{12}$.

Let $r_X$ be the Cartier index of $K_X$. By {\bf Step 2} and the choice of $P$, $12$ is divisible by $r_X$. As $r_XK_X^3$ is a positive integer, we have $K_X^3\geq\frac{1}{12}$.

{\bf Step 4.} We finish the proof in this step.

Since $F_X\sim 3K_X$, we have
\begin{align*}
    K_X^3=\frac{1}{3}(K_X^2\cdot F_X)=\frac{1}{3}(\pi^*(K_X)^2\cdot F)=\frac{1}{3}(\pi^*(K_X)|_F)^2=\frac{1}{36},
\end{align*}
which is a contradiction by {\bf Step 3}. The proof is completed.

\end{proof}
The following proposition follows by Proposition \ref{prop: birationality (1,2) surface 1}, Proposition \ref{prop: birationality (1,2) surface 2} and Proposition \ref{prop: index of X}.
\begin{prop}\label{prop: birationality (1,2)}
    Let $X$ be a regular minimal $3$-fold of general type with $P_3(X)\geq 2$. Suppose that $F$ is a $(1,2)$-surface. Then $\varphi_m$ is birational for all $m\geq 14$.
\end{prop}

By Lemma \ref{lem: irregular}, Proposition \ref{prop: mK vol(F)>=2}, Proposition \ref{prop: mK (1,0)}, Proposition \ref{prop: mK (1,1)} and Proposition \ref{prop: birationality (1,2)}, we have the following theorem.
\begin{thm}\label{thm: birational 14}
    Let $X$ be a minimal $3$-fold of general type with $P_3(X)\geq 2$. Then $\varphi_m$ is birational onto its image for all $m\geq 14$.
\end{thm}

\subsection{Characterization}
In this section, we will characterize the birational structure of minimal $3$-folds $X$ of general type which satisfy the following assumptions:

{\bf ($\pounds$)} {\em $P_3(X)\geq 2$ and $\varphi_{13}$ is non-birational. }

\begin{lem}\label{lem: 3K is pencil}
    Let $X$ be a minimal $3$-fold of general type with $P_3(X) \geq 2$. Suppose that $X$ satisfies ($\pounds$). Then $q(X)=0$ and $|3K_X|$ is composed with a pencil of $(1,2)$-surfaces. 
\end{lem}

\begin{proof}
     We have $q(X)=0$ by Lemma \ref{lem: irregular}. Denote by $d_3$ the dimension of the image of the $3$-canonical map. 
     
     If $d_3=3$, by \cite[Theorem 2.20]{EXP2}, $\varphi_m$ is birational for all $m\geq 10$, which is a contradiction. If $d_3=2$, by  \cite[Theorem 2.22]{EXP2}, $\varphi_m$ is birational for all $m\geq 13$, which is a contradiction. It follows that $|3K_X|$ is composed with a pencil. 

     By Proposition \ref{prop: mK vol(F)>=2}, Proposition \ref{prop: mK (1,0)} and Proposition \ref{prop: mK (1,1)}, we deduce that $|3K_X|$ is composed with a pencil of $(1,2)$-surfaces.
\end{proof}

So we only need to consider the case when $q(X)=0$ and $|3K_X|$ is composed with a pencil of $(1,2)$-surfaces. Take $\pi: X'\to X$ as in \S \ref{setup}. It is clear that $f: X'\to \Gamma$ and $|3K_{X'}|$ are composed with the same pencil. We have $\Gamma\cong \mathbb{P}^1$ since $q(X)=0$.

\begin{prop}\label{prop: birational 13 (1,2)-surface}
    Let $X$ be a regular minimal $3$-fold of general type with $P_3(X)\geq 3$. Suppose that $|3K_X|$ is composed with a pencil of $(1,2)$-surfaces. Then $\varphi_{m}$ is birational for all $m\geq 13$.
\end{prop}
\begin{proof}
Let $m\geq 13$ be an integer. We have $\pi^*(3K_X)\geq 2F$. By \cite[Corollary 2.3]{Noether} (take $\lambda=\frac{3}{2}$, $D=K_X$ and $S=F$), there is an effective $\bQ$-divisor $H_1$ such that
\begin{align}\label{eq: restriction(P_3>2)}
    \pi^*(K_X)|_F\simQ \frac{2}{5}\sigma^*(K_{F_0})+H_1.
\end{align}
Following \cite[Subcase 3.4.2]{EXP3} (take $m_0=3$ and $\theta=2$), we have
\begin{align}\label{eq: xi*(P_3>2)}
\xi \geq \frac{1}{2}.
\end{align}
By \eqref{eq: restriction(P_3>2)}, there exists an effective $\bQ$-divisor $\tilde{E}_F$ such that 
\begin{align}\label{eq: restriction 2(P_3>2)}
    \pi^*(K_X)|_F\simQ \frac{2}{5}C+\tilde{E}_F.
\end{align}
From \eqref{eq: restriction of linear system}, we have 
$$
|mK_{X'}||_F\lsgeq |K_F+\roundup{L_m|_F}|.
$$
By \eqref{eq: restriction 2(P_3>2)},  it is clear that
$$
L_m|_F-\frac{5}{2}\tilde{E}_F-C\equiv (m-\frac{13}{2})\pi^*(K_X)|_F
$$
is nef and big. By Kawamata-Viehweg vanishing theorem, we have
\begin{align*}
    |K_F+\roundup{L_m|_F}||_C&\lsgeq |K_F+\roundup{L_m|_F-\frac{5}{2}\tilde{E}_F}||_C\\
    &\lsgeq |K_C+\roundup{L_m|_C-\frac{5}{2}{\tilde{E}_F}|_C}|.
\end{align*}
Note that we have
$$
\deg (L_m|_C-\frac{5}{2}{\tilde{E}_F}|_C)=(m-\frac{13}{2})(\pi^*(K_X)\cdot C)=(m-\frac{13}{2})\xi>2,
$$
where the last inequality holds since $m\geq 13$ and $\xi\geq\frac{1}{2}$. It follows that the linear system $|K_C+\roundup{L_m|_C-\frac{5}{2}{\tilde{E}_F}|_C}|$ gives a birational map. By Theorem \ref{thm: bir. prin.}, Lemma \ref{lem: disting. g.i.e s} and Lemma \ref{lem: separate genus 2 curve}, $\varphi_m$ is birational for all $m\geq 13$. The proof is completed.
\end{proof}

\subsection{Proof of Theorem \ref{main theorem}} (1) is just Theorem \ref{thm: birational 14}. (2) follows by Lemma \ref{lem: 3K is pencil} and Proposition \ref{prop: birational 13 (1,2)-surface}.

\section*{\bf Acknowledgment}

Both authors would like to thank Professor Meng Chen for the suggestion of this topic. Y.H. would like to thank Professors Chen Jiang and Tong Zhang for fruitful discussions. The first author was supported by National Key Research and Development Program of China \#2023YFA1010600 and the National Natural Science Foundation of China (Grant No. 12201397).



\begin{thebibliography}{999999}

\bibitem{BCHM}  C.  Birkar, P. Cascini, C. D. Hacon and J. McKernan, {\em Existence of minimal models for varieties of log general type}, J. Amer. Math. Soc. {\bf 23} (2010), no. 2, 405--468.  %

\bibitem{Bom} E. Bombieri, {\em Canonical models of surfaces of general type}, Inst. Hautes \'Etudes Sci. Publ. Math. {\bf 42} (1973), 171--219. %

\bibitem{CCCJ} J.-J. Chen, J. A. Chen, M. Chen and Z. Jiang, {\em On quint-canonical birationality of irregular threefolds}, Proc. Lond. Math. Soc. (3) {\bf 122} (2021), no. 2, 234--258. %

\bibitem{EXP1} J. A. Chen and M. Chen, {\em Explicit birational geometry of threefolds of general type, I}, Ann. Sci. \'Ec. Norm. Sup\'er. {\bf 43} (2010), 365--394.

\bibitem{EXP2} J. A. Chen and M. Chen, {\em Explicit birational geometry of threefolds of general type, II}, J. Differ. Geom. {\bf 86} (2010), 237--271.%

\bibitem{EXP3} J. A. Chen and M. Chen, {\em Explicit birational geometry for 3-folds and 4-folds of general type, III}, Compos. Math. {\bf 151} (2015), 1041--1082.%

\bibitem{Noether} J.~A.~Chen, M.~Chen, C.~Jiang, {\it The Noether inequality for algebraic $3$-folds}, with an appendix by J. Koll\'ar, Duke Math. J. {\bf 169} (2020), no. 9, 1603--1645. %

\bibitem{CCJ} J. A. Chen, M. Chen, Z. Jiang, {\em On 6-canonical map of irregular threefolds of general type}, Math. Res. Lett. {\bf 20} (2013), 33--39. %

\bibitem{CH} J. A. Chen, C. D. Hacon, {\em Pluricanonical systems on irregular 3-folds of general type}, Math. Z. {\bf 255} (2007), no. 2, 343--355. %

\bibitem{Chen03} M. Chen, {\em Canonical stability of 3-folds of general type with $p_g \geq 3$}, Internat. J. Math. {\bf 14} (2003), 515--528. %

\bibitem{Chen18} M. Chen, {\em On minimal 3-folds of general type with maximal pluricanonical section index}, Asian J. Math. {\bf 22} (2018), no. 2, 257--268. %

\bibitem{CHJ} M. Chen, Y. Hu and C. Jiang, {\em On moduli spaces of canonical threefolds with small genera and minimal volumes}, arXiv: 2407.01276. %

\bibitem{HM} C. D. Hacon and J. McKernan, {\em Boundedness of pluricanonical maps of varieties of general type}, Invent. Math. {\bf 166} (2006), 1--25. %

\bibitem{IF} A.R. Iano-Fletcher, {\em Working with weighted complete intersections}, in: Explicit birational geometry of 3-folds, London Mathematical Society Lecture Note Series, vol. 281 (Cambridge University Press, Cambridge, 2000), 101--173. %

\bibitem{Kollar} J. Koll\'ar, {\em Higher direct images of dualizing sheaves I}, Ann. Math. {\bf 123} (1986), no.1, 11--42. %

\bibitem{KM} J. Koll\'ar and S. Mori, {\em Birational geometry of algebraic varieties}, Cambridge Tracts in Mathematics, {\bf 134}, Cambridge University Press, Cambridge, 1998. %

\bibitem{Miyaoka} Y. Miyaoka, {\em Tricanonical maps of numerical Godeaux surfaces}, Invent. Math. {\bf 34} (1976), no.2, 99--111. %

\bibitem{Takayama} S. Takayama, {\em Pluricanonical systems on algebraic varieties of general type}, Invent. Math. {\bf 165} (2006), 551--587. %

\bibitem{Tsuji}  H. Tsuji, {\em Pluricanonical systems of projective varieties of general type. I}, Osaka J. Math. {\bf 43} (2006), no. 4, 967--995. %


\end{thebibliography}
\end{document}